\documentclass{amsart}
\usepackage{amssymb,amsmath,amsrefs}
\usepackage[colorlinks=true]{hyperref}
\hypersetup{urlcolor=blue, citecolor=green}
\usepackage[dvips]{graphicx}
\usepackage{color}

\usepackage{float}
\usepackage{xcolor}
\theoremstyle{plain}

\newtheorem*{conj*}{Conjecture}
\newtheorem*{cor*}{Corollary}
\newtheorem{theorem}{Theorem}[section]

\newtheorem{proposition}[theorem]{Proposition}
\newtheorem{corollary}[theorem]{Corollary}
\newtheorem{lemma}[theorem]{Lemma}

\newtheorem{claim}{Claim}

\theoremstyle{definition}
\newtheorem*{def*}{Definition}
\newtheorem{remark}[theorem]{Remark}
\newtheorem{rmk}[theorem]{Remark}
\newtheorem{example}[theorem]{Example}

\newtheorem{definition}[theorem]{Definition}

\newcommand{\T}{\mathbb{T}}

\newcommand{\be} {\beta}        
\newcommand{\ga} {\gamma}    
\newcommand{\de} {\delta}

\renewcommand{\epsilon}{\varepsilon}

\newcommand{\Z}{\mathbb{Z}}
\newcommand{\N}{\mathbb{N}}
\newcommand{\R}{\mathbb{R}}
\newcommand{\eps}{\varepsilon}

\newcommand{\per}{\operatorname{Per}}

\newcommand{\sing}{\operatorname{Sing}}

\title{Rescaled-Expansive Flows: Unstable Sets and Topological Entropy}

\author{Alexander Arbieto}
\address{Instituto de Matem\'atica, Universidade Federal do Rio de Janeiro, P. O. Box 68530, 21945-970 Rio de Janeiro, Brazil.}

\email{arbieto@im.ufrj.br}

\author{Alfonso Artigue}

\address{Departamento de Matematica y Estad\'istica del Litoral, Universidad de la Rep\'ublica, Gral. Rivera 1350, Salto, Uruguay}

\email{artigue@unorte.edu.uy}

\author{Elias Rego}

\address{Department of Mathematics, Southern University of Science and Technology, Shenzhen, China.}

\email{rego@sustech.edu.cn}

\thanks{Primary MSC code: 37B05, secundary MSC code:37C70. Keywords: rescaled-expansiveness, topological entropy, attractors. 
A. Arbieto was partially supported by CNPq, FAPERJ and
PRONEX/DS from Brazil, E. Rego  was partially supported by  NSFC 12250710130 from China.
}

\begin{document}

\maketitle

\begin{abstract}
In this work, we introduce and explore a rescaled theory of local stable and unstable sets for rescaled-expansive flows and its applications to topological entropy. We introduce a rescaled version of the local unstable sets and the unstable points. We find conditions for points of the phase space to exhibit non-trivial connected pieces of such unstable sets. We apply these results to the problem of proving positive topological entropy for rescaled-expansive flows with non-singular Lyapunov stable sets.

\end{abstract}

\section{introduction}

The property of expansiveness introduced by R. Utz in 1950  is a landmark of the dynamical systems theory. Its great success is in part due to its proximity to the hyperbolic theory and its close relationship with many important topics of the dynamical systems theory, such as the stability theory and the entropy theory. Very soon, expansiveness was perceived as a source of complex dynamical behavior. Indeed, many expansive systems exhibit chaotic features. We refer the reader to  \cite{AH} for a detailed exposition of the dynamical properties of expansive homeomorphisms. 

The concept of expansive flow  was introduced  in \cite{BW} by R. Bowen and P. Walters
to describe the behavior of axiom A flows, but it is not appropriate to deal with flows exhibiting singularities accumulated by regular orbits, such as the Lorenz attractor.
For these flows,
M. Komuro introduced in \cite{K1} the concept of $k^*$-expansiveness. Later, other versions of expansiveness were introduced \cite{Ar1,Ar2,Ar3,WW}.

Several authors have been interested in the chaotic behavior of expansive flows and topological entropy is among the most important
ways of measuring the complexity of a dynamical system. Indeed, positive entropy is an indicative of chaotic behavior.  In many contexts, expansiveness is related to positive topological entropy. For instance,  in \cite{Fa} A. Fathi showed that any expansive homeomorphism has positive topological entropy if the phase space has  positive topological dimension.
H. Kato generalized this in \cite{Ka} for continuum-wise expansiveness (by applying techniques developed by
R. Ma\~{n}\'{e} in \cite{Ma}) and for expansive flows in \cite{ACP} by A. Arbieto, W. Cordeiro and M.J. Pac\'{i}fico.
These results are of topological character as they do not assume any differentiable structure on the system's phase space.

To our best knowledge, there is no result of topological nature regarding the positiveness of the
entropy of a $k^*$-expansive flow.
In this work, we aim to explore the implications of  expansiveness to the topological entropy of expansive singular flows.
The main reasons for this lack of results are the following:

\begin{enumerate}    
    \item The existence of many distinct versions of expansiveness for singular flows.
    \item These results are strongly dependent on the uniform expansiveness property, but expansive singular flows may not satisfy this property. 
    \item The nonexistence of cross-sections for singularities and the loss of control over the size and the time of the cross-sections at regular points.
    
\end{enumerate}
Actually,
the techniques of \cite{Ka,ACP}
to prove positive entropy
are strongly supported by the existence of non-trivial connected local stable or unstable sets.
Unfortunately, the above-listed facts may forbid the existence of such local stable sets.
Some examples will be considered in Section \ref{examples}.

This article deals with the rescaled-expansiveness property (R-expansiveness for short) introduced by L. Wen and X. Wen in \cite{WW}. Our goal is to study the topological entropy of R-expansive flows. To achieve it, we introduce a rescaled version of local stable and unstable sets for singular flows based on the dynamics of the holonomy maps along orbits of regular points. We obtain conditions to R-expansive flows that admit non-trivial pieces of R-unstable sets with "hyperbolic behavior" and study their influence in the topological entropy of R-expansive flows.  This choice of expansiveness is made due to its closeness with $k^*$-expansiveness, but also due to its suitability to work with flow boxes, providing us a nice control over the holonomy maps between cross-sections which is essential to the study of stable and unstable sets. 

This text is organized as follows: In section \ref{prel}, we establish the basic notation and the primary setting used through this work. In section \ref{sectionexp}, we study R-expansiveness in more detail and explore its role in the existence of local stable/unstable sets. Section \ref{entropysection} is devoted to studying the entropy of R-expansive flows, providing the reader with some new examples of R-expansive flows and explaining how they are related to our results.

\section{Preliminaries}\label{prel}
This section is devoted to establishing the basic setting we will work on.
Throughout this paper, $M$ denotes a compact and boundary-less smooth Riemannian manifold.  Let us denote $g$ for the Riemannian metric of  $M$.  In addition, we denote $d$  and $\|\cdot \|$ for the distance induced on $M$ and the norm induced on $TM$ by the Riemannian metric  $g$, respectively.

\begin{definition}
 A $C^r$-\textit{flow} $\phi$ on $M$ is a $C^r$-map $\phi:\R\times M \to M$ satisfying the following conditions:
\begin{enumerate}
    \item $\phi(0,x)=x$, for every $x\in M$.
    \item $\phi(t+s,x)=\phi(t,\phi(s,x)))$, for every $t,s\in \R$ and every $x\in M$. 
\end{enumerate}
\end{definition}

Throughout this work, we will always assume $r\geq 1$. In this case, $\phi$ generates a velocity vector field that will be denoted by $X$. Let us denote by $\phi_t$ the map $\phi(t,\cdot):M\to M$, when $t$ is fixed. The \textit{orbit} and the \textit{positive orbit}  of a point  $x$ are, respectively, the sets $$O(x)=\{\phi_t(x); t\in \R\} \textrm{ and } O_+(x)=\{\phi_t(x); t\geq 0\}.  $$
We say that $x\in M$ is a \textit{singularity}
if $\phi_t(x)=x$ for all $t\in\R$.
A point $x\in M$ is \textit{periodic} if it is not a singularity and there exists $t>0$ such that $\phi_t(x)=x$. The sets of singularities and periodic points are denoted by $\sing(\phi)$ and $\per(\phi)$, respectively.  We say that a set $\Lambda$ is \textit{invariant} is $\phi_t(\Lambda)=\Lambda$, for every $t\in \R$.

\begin{definition}
Let $\Lambda$ be a compact and invariant set.
We say that $\Lambda$ is \textit{Lyapunov stable} if for any $\eps>0$, there is some $\de>0$ such that if $x\in B_{\de}(\Lambda)$, then $\phi_t(x)\in B_{\eps}(\Lambda)$, for every $t\geq 0$.
We say that $\Lambda$ is an \textit{attractor} if:
\begin{enumerate}
\item $\phi|_{\Lambda}$ is \textit{transitive}, \textit{i.e.}, there is $x\in \Lambda$ such that $\overline{O_+(x)}=\Lambda$.
\item There is an open neighborhood  $U$ of $\Lambda$ satisfying:
    \begin{enumerate}
    \item $\overline{\phi_t(U)}\subset U$ for any $t>0$ and
    \item  $\Lambda=\cap_{t\geq0}\phi_t(U)$.
    \end{enumerate}
\end{enumerate}
\end{definition}

Every attractor is Lyapunov stable, but the converse does not hold.
A neighborhood of $\Lambda$ as in the above definition is called \textit{isolating neighborhood}.
We say that an attractor $\Lambda$ is \textit{non-periodic} if it is not a periodic orbit. Let $\Lambda\subset M$ be compact and invariant. 

Now, we recall the definition of topological entropy for flows. Fix $\eps>0$ and $t>0$.
We say that a pair of points is \textit{$t$-$\eps$-separated by $\phi$} if there is some $0\leq s\leq t$ such that $d(\phi_s(x),\phi_s(y))>\eps$.
Also, a subset $E\subset M$ is $t$-$\eps$-\textit{separated}
if any pair of distinct points of $E$
is $t$-$\eps$-separated by $\phi$.
For $\Lambda\subset M$,
let $s_t(\eps,\Lambda)$ denote the maximal cardinality of a $t$-$\eps$-separated subset of $\Lambda$. This number is finite due to the compactness of $M$. We define the \textit{topological entropy} of $\phi$ on $\Lambda$ to be the number $h(\phi,\Lambda)$ defined by
 $$ h(\phi,\Lambda)=\lim_{\eps\to 0}\limsup_{t\to \infty}\frac{1}{t}\log s_t(\eps,\Lambda).$$

\begin{definition}
The \textit{topological entropy} of $\phi$ is  the number $h(\phi)=h(\phi,M)$.
\end{definition}

The problem of finding positive topological entropy for expansive systems was first considered in the 80's and 90's by Fathi, Kato, and Lewowicz (see) \cites{Fa,Ka,Her}.
Its version for expansive flows is
proved by A. Arbieto, W. Cordeiro, and M. J. Pac\'{i}fico in \cite{ACP}.

\subsection{Expansiveness}

We start by giving the definition of expansiveness, which R. Bowen and P. Walters introduced in \cite{BW}.

\begin{definition}
% [Expansiveness]
A flow  $\phi$ is \textit{expansive} if for every $\eps>0$, there is $\de>0$ such that the following holds:
If $x,y\in M$,  $\rho\colon\R\to \R$ is a continuous function satisfying $\rho(0)=0$ and  $$d(\phi_t(x),\phi_{\rho(t)}(y))\leq \de$$ for every $t\in \R$, then $y\in\phi|_{[-\eps,\eps]}(x)$. We say that a compact and invariant set $\Lambda\subset M$ is expansive if the flow restricted to $\Lambda$ is expansive.
\end{definition}

\begin{theorem}[\cite{ACP}]\label{EntCW}
Let $\phi$ be a continuous flow and suppose $\dim(M)>1$. If $\phi$ is expansive then $h(\phi)>0$.
\end{theorem}

\begin{rmk}The previous result was proved in the context of $CW$-expansive flows, which, in turn, contains the expansive flows. In this work, we will not detail $CW$-expansive flows since this concept will not be addressed here. 

\end{rmk}

The property of expansiveness was designed to be the model for the expansive behavior displayed by axiom A and Anosov flows.
Unfortunately, this concept does not capture the singular behavior of flows such as the Lorenz Attractor.  To cover these flows, M. Komuro introduced in \cite{K1} the following concept:

\begin{definition}
% [$k^*$-Expansiveness]
A flow $\phi$ is $k^*$-\textit{expansive} if for every $\eps>0$, there is some $\de>0$ such that the following holds:
if $x,y\in M$,  $\rho:\R\to \R$  is an  increasing homemorphism  and  $$d(\phi_t(x),\phi_{\rho(t)}(y))\leq \de$$ for every $t\in \R$, then there are  $t_0,s\in \R$ such that $|s|\leq \eps$ and $\phi_{\rho(t_0)}(y)=\phi_{t_0+s}(x)$. We say that a compact and invariant set $\Lambda\subset M$ is $k^*$-expansive if the flow restricted to $\Lambda$ is $k^*$-expansive.
\end{definition}

\begin{rmk}
 Expansive flows do not exist in surfaces (see \cite{BW} for details), while in \cite{Ar3} surface, $k^*$-expansive flows are classified. On the other hand, it was proved in \cite{Y} that every surface flow has zero topological entropy. Thus, there is no result similar to Theorem \ref{EntCW} for $k^*$-expansiveness for $\dim(M)=2$.
\end{rmk}

\section{Rescaled expansiveness}\label{sectionexp} In this section, we give the definition of the main concept used in this work: The rescaled-expansiveness and study its influence on the existence of non-trivial local R-stable and R-unstable sets. 

\subsection{Cross sections and flow boxes}
The techniques used in \cite{ACP} to obtain positive entropy for expansive flows are strongly supported by the fact that if $\phi$ is non-singular, then one can always choose cross-sections for every point with uniform size (see \cite{BW}). This allows us to use flow boxes with uniform size to study the dynamics. Unfortunately, it does not hold for singular flows. To see this, let $x\in M$ be a regular point for $\phi$.  The normal space of $x$ in $T_xM$ is the set  $$\mathcal{N}(x)=\{v\in T_xM; v \perp X(x)\}.$$
Let us denote $\mathcal{N}_{r}(x)=\mathcal{N}(x)\cap \mathcal{B}_{r}(0)$, where $\mathcal{B}_{r}(0)$ is the ball in $T_xM$ of radius $r$ and centered at $0$. The tubular flow theorem for smooth flows asserts that for any regular point $x$, there are $\eta_x>0$ and $r_x>0$ such that the set $$N_{r_x}(x)=\exp_x(\mathcal{N}_{r_x}(x)) $$ is a cross-section of time $\eta_x$ through $x$, \textit{i.e.}, for any $y\in N_{r_x}(x)$ we have that $$\phi_{[-\eta_x,\eta_x]}(y)\cap N_{r_x}(x)=\{y\}.$$
Furthermore, any $y\in N_{r_x}(x)$ is regular. 
Another important consequence of the tubular flow theorem is that it allows us to work with \textit{holonomy maps} generated by flows. To make the previous assertion precise, let $x\in M$ be a regular point, fix some $t\in\R$ and suppose that $N_{\eps}(\phi_t(x))$ is a cross-section. Then, by the tubular flow theorem, there are $r_x>0$, and
a continuous function $\tau\colon N_{r_x}(x)\to\R$ such that
$\phi_{\tau(y)}(y)\in N_{r_x}(\phi_t(x))$ for all $y\in N_{r_x}(x)$ and $\tau(x)=t$.
In this way, we define the holonomy map $$P_{x,t}:N_{r_x}(x)\to N_{\epsilon}(\phi_{t}(x))$$ by setting $P_{x,t}(y)=\phi_{\tau(y)}(y)$.

One of the main difficulties in the use of cross-sections and holonomy maps for singular flows is that the radius $r_x$ may go to zero when $x$ approaches some singularity. The next result allows us to have a better control on these cross-sections. Indeed, it gives us an explicit relation between the constants $\eps$ and $r_x$ used above.
Before stating the result, let us fix the following notation that will be used throughout this
paper $$N^r_{\eps}(x)=N_{\eps\|X(x)\|}(x).$$

\begin{theorem}[\cite{WW}]\label{RFB}
Suppose that $X$ is a $C^1$-vector field and let $\phi$ be the flow induced by $X$. Then there exist $L>0$ and $\be_0>0$ such that for any $0<\be<\be_0$, $t>0$ and $x\in M \setminus \sing(\phi)$ we have:
 \begin{enumerate}
\item The set $\phi|_{[-\be,\be]}(N^r_{\be}(x))$ is a flow box; in particular, it does not contain singularities.
\item The ball $B_{\frac{1}{3}\be||X(x)||}(x)$ is contained on  $\phi|_{[-\be,\be]}(N^r_{\be}(x))$

\item The holonomy map   $$P_{x,t}:N_{\frac{\be}{L^{t}}}^r(x) \to N^r_{\be}(\phi_t(x))$$ is well defined and injective. Moreover, for any $y\in N_{\frac{\be}{L^{t}}}^r(x)$ we have $$d(\phi_s(x),\phi_s(y))\leq \be|| X(\phi_s(x))||$$ for any $0\leq s\leq t$. The same statement is valid for $t<0$.  

\end{enumerate}

\end{theorem}

\subsection{R-expansiveness}
Based on the above ideas, L. Wen and X. Wen in \cite{WW} introduced a \textit{rescaled} version of expansiveness
by considering that

the distance of separation of the orbits is rescaled by the size of the velocity vector field.

\begin{definition}
A $C^r$-flow $\phi$ on $M$ is \textit{R-expansive} (or \textit{rescaled expansive}) if for every $\eps>0$ there is some $\de>0$ such that:
if $x,y\in M$, $\rho:\R\to \R$  is a increasing continuous function and $$d(\phi_t(x),\phi_{\rho(t)}(y))\leq \de||X(\phi_t(x))  ||$$ for every $t\in \R$, then $\phi_{\rho(t)}(y)\in \phi|_{[t-\eps,t+\eps]}(x)$ for any $t\in \R$. We say that a compact and invariant set $\Lambda\subset M$ is R-expansive if the flow restricted to $\Lambda$ is R-expansive.
\end{definition}

In \cite{Ar2} it was proven that $k^*$-expansiveness implies R-expansiveness under the assumption of hyperbolicity of $Sing(\phi)$. Later, this result was improved into the following:

\begin{theorem}[\cite{RWY}]\label{artiguerelation}
  If $\phi$ is $k^*$-expansive,  then $\phi$ is R-expansive.
\end{theorem}

\subsection{Stable and Unstable sets}
 Here we give a new definition of local stable and unstable sets for regular points of a singular flow, based on holonomy maps and rescaled distances. Fix some regular point $x\in M$.  By Theorem \ref{RFB}  for any $0<\be\leq \be_0$  the set  $N_{\be}^r(x)$ is a cross section of radius $\be||X(x)||$ for the flow.  Moreover, for any $t>0$ the holonomy map $P_{x,t}$ is well defined on $N^r_{\frac{\be}{L^t}}(x)$ and if $y\in N^r_{\frac{\be}{L^t}}(x)$,  the orbit segment  between $y$ and $P_{x,t}(y)$ belongs to the $\be$-rescaled  tubular neighborhood of   $O(x)$.   Let us fix $x\in M\setminus \sing(\phi)$, $t>0$  and $\be>0$.

\begin{definition}
The $R$-\textit{stable} and $R$-\textit{unstable local sets} of $x$ with size $\beta$ and time $t$ are, respectively
$$S_{\be}(t,x)=\left\{y\in N^r_{\frac{\be}{L^t}}\left(x\right); d(P_{x,nt}(x),P_{x,nt}(y))\leq \frac{\be}{L^t}||X(P_{x,nt}(x)) ||, \forall n\in \N\right\}$$ 
$$U_{\be}(t,x)=\left\{y\in  N^r_{\frac{\be}{L^t}}\left(x\right); d(P_{x,-nt}(x),P_{x,-nt}(y))\leq \frac{\be}{L^t}||X(P_{x,-nt}(x)) ||, \forall n\in \N\right\}$$ 
\end{definition}

Let us present a useful characterization of R-expansiveness in terms of R-stable and R-unstable sets.

\begin{proposition}\label{charac}
The flow $\phi$ is R-expansive if, and only if,  there exists $\de>0$ such that for any regular point $x\in M$  and any $t>0$, one has $S_{\de}(t,x)\cap U_{\de}(t,x)=\{x\}$.
\end{proposition}

\begin{proof}
Let $x$ be a regular point, fix $\be>0$ small enough and let $0<\eps<\be$. Let $0<\de'<\eps$ be given by the R-expansiveness of $\phi$ related to $\eps$.
Fix $\delta=\frac{\delta'}{3}$ and suppose  $$y\in S_{\de}(t,x)\cap U_{\de}(t,x).$$ 
Next, we shall construct a reparametrization $\rho$ satisfying: 
 $$d(\phi_s(x),\phi_{\rho(s)}(y))\leq \de||(X(\phi_t(x))||,$$
for every $s\in \mathbb{Z}.$
Since $y\in S_{\de}(t,x)$, we have $$ d(P_{x,nt}(x),P_{x,nt}(y))<\frac{\de}{L^t}||X(P_{x,nt}(x))||$$ for every $n\geq 0$. 
  Thus,  from to Theorem \ref{RFB},  we obtain the following facts:
\begin{enumerate}  
\item For every $s\geq 0$,  $N^r_{\delta}(\phi_{s}(x))$ is a cross section through $\phi_s(x)$. 
\item For every $n\geq 0$ and every $s \in[0,1]$, it holds  $$d(\phi_s(P_{x,nt})(x)),\phi_s(P_{x,nt}(y)))\leq\delta||X(\phi_s(P_{x,nt}(x))) ||.$$ 
\item $B^r_{\delta}(P_{x,nt}(x))$ is contained in the flowbox of $N^r_{\delta'}(P_{x,nt}(x))$.  

\end{enumerate}

Define $\rho_+:[0,+\infty)\to [0,-\infty)$ by parts as follows:\\  If $n\geq 0$ and $s\in [nt,(n+1)t)$, then $\rho_+(s)$  is the unique time such that:
\begin{enumerate}   
 \item $\phi_{\rho_+(s)}(y)$ belongs the orbit segment from $P_{x,nt}(y)$ to $P_{x,(n+)t}(y)$.  
\item $\phi_{\rho_+(s)}(y)\in N^r_{\delta}(\phi_s(x))$.
\end{enumerate}

 \begin{figure}[h]
     \centering
   \includegraphics[scale=0.4]{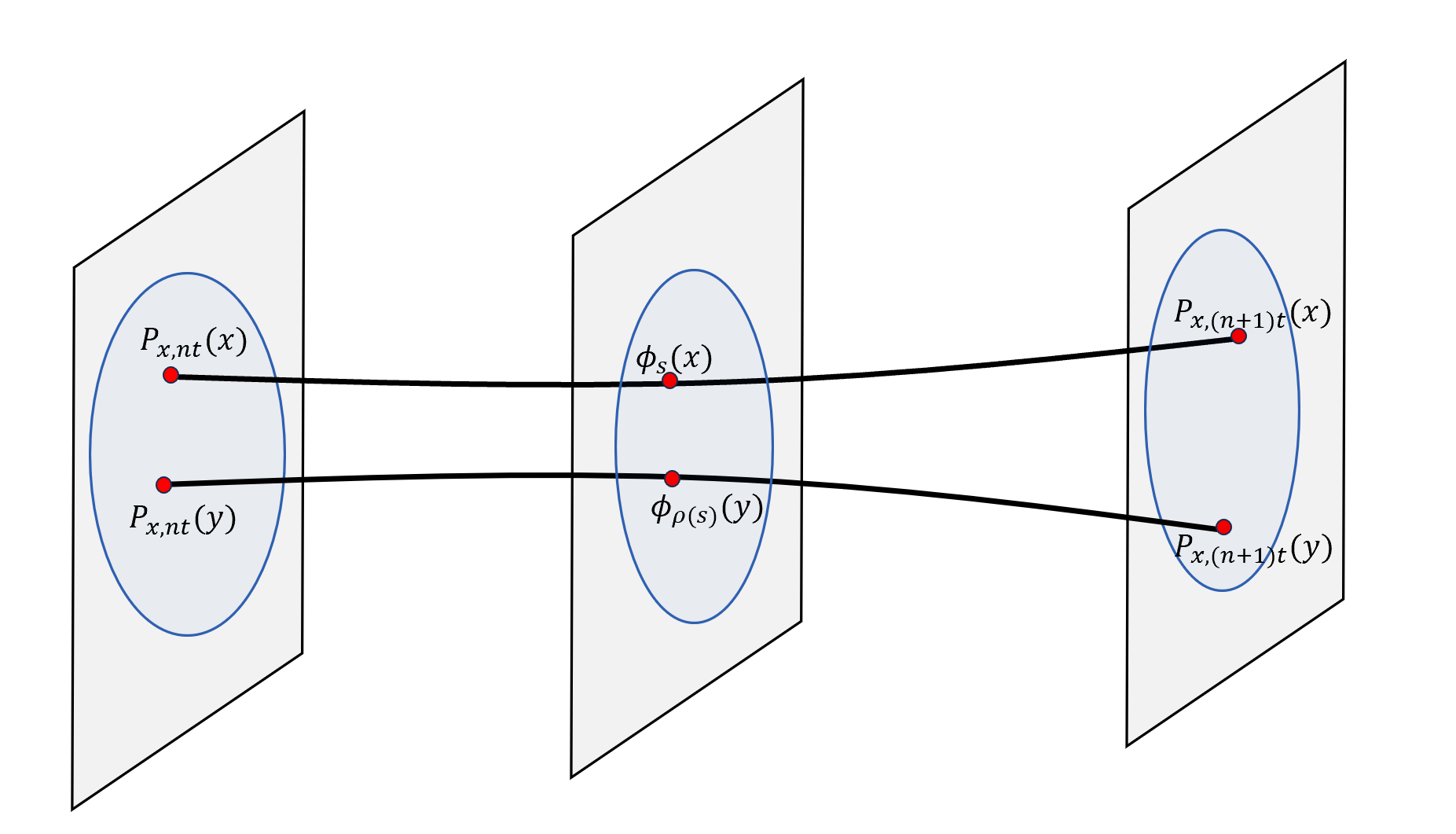}
   \caption{The construction of $\rho$.}
    \label{fig:my_label}
 \end{figure}

It is easy to see that $\rho_+$ is a homeomorphism and satisfies
$$d(\phi_s(x),\phi_{\rho(s)}(y))\leq\delta||X(\phi_s)(x) ||, $$
for every $s\geq0$.  

 By an analogous reasoning, since $x\in \mathcal{U}_{\delta}(t,x)$ we can construct a homeomorphism $$\rho_-:(-\infty,0]\to (-\infty,0]$$ such that $$d(\phi_s(x),\phi_{\rho(s)}(y))\leq\delta||X(\phi_s)(x) ||, $$
for every $s\leq0$.  Finally, we define:
$$\rho(s)=\begin{cases} \rho_+(s), \textrm{ if } s\geq 0\\
\rho_-(s), \textrm { if } s\leq0.
\end{cases}
  $$
 Now R-expansiveness implies that $y\in \phi_{[-\eps,\eps]}(x)$, but since $x,y\in N^{r}_{\frac{\de}{L^t}}(x)$,  we have that $y=x$.

Conversely, denote  $$B=\sup\{||X(x)||;x\in M\},$$  fix $0<\eps<\be$ and let $0<\de<\eps$ be such that $$S_{\frac{\de}{B}}(t,x)\cap U_{\frac{\de}{B}}(t,x)=\{x\}$$ for any regular point $x$ and any $t>0$. Fix $t>0$ such that $L^t>1$ and suppose there exist a reparametrization $\rho$ and two points $x,y$ satisfying $$ d(\phi_s(x),\phi_{\rho(s)}(y))\leq \frac{\de}{BL^t}||X(\phi_s)(x)) ||$$ for $ s\in \R$. This implies in particular that $$d(x,y)<\frac{\de}{BL^t}||X(x)||.$$ Since $\de<\eps<\be$, Theorem \ref{RFB} implies that there exists some $$|s_0|<\frac{\de}{BL^t}||X(x)||\leq\de\leq\eps$$ such that $y_0=\phi_{s_0}(y)\in N^r_{\de}(x)$. More generally, since $$d(\phi_{nt}(x),\phi_{h(nt)}(y))<\frac{\de}{BL^t}||X(\phi_{nt}(x)||,$$ for any $n\in \Z$, there exists $|s_n|<\eps$ such that $$y_n=\phi_{\rho(nt)+s_n}(y)\in N^r_{\de}(\phi_{nt}(x)).$$  But the last fact implies that the set $\{y_n\}$ is the orbit of $y_0$ under the holonomy maps $\{P_{x,nt}\}$. In additon, one has that $y_0\in S_{\de}(t,x)\cap U_{\de}(t,x)$  and therefore we must have $y_0=x$. Then $\phi_{s_0}(y)=x$ and the flow $\phi$ is R-expansive.
\end{proof}

An interesting fact about the above characterization is that we do not need to be concerned about reparametrizations since we are only working with the holonomy maps generated by $\phi$.

\subsection{Uniformity}
For the remainder of this section,  we are assuming that the flows in consideration are R-expansive, and the constant $\de$ given by Proposition  \ref{charac} will be called a constant of R-expansiveness of $\phi$.

\begin{remark}
    Notice that if $\delta$ is a $R$-expansiveness constant for $\phi$, then any $0<\delta'<\delta$ is also an $R$-expansiveness constant for $\phi$.
\end{remark}

Next, we work in order to obtain versions of some well known results about non-singular expansive flows to the R-expansive case. For $\Lambda\subset M$, we denote $$A_\Lambda=\inf_{x\in \Lambda}\{||X(x)||\}.$$ 

\begin{theorem}[Uniform R-expansiveness]\label{uef}
Suppose $\phi$ is R-expansive with constant of R-expansiveness $\de$ and let $\Lambda\subset M$ be a non-singular, compact and invariant set.
Then for any $0<\eta\leq \de A_\Lambda$ and  $t>0$, there exists $J>0$ such that if $x\in \Lambda$ and $y\in N^r_{\de}(x)$ with $d(x,y)>\eta$ , then there is $-J\leq i\leq J$ such that $$d(P_{x,it}(x),P_{x,it}(y))\geq \de||X(P_{x,it}(x))||.$$
\end{theorem}

\begin{proof}
Suppose the result is false. Thus there are $\eta>0$, $t>0$, sequences  $x_n\in \Lambda$, $y_n\in N^r_{\de}(x_n)$, $m_n\to \infty$,  such that $d(x_n,y_n)>\eta$ and $$d(P_{x_n,it}(x_n),P_{x_n,it}(y_n))\leq \de||X(P_{x_n,it}(x_n)||$$ for $-m_n\leq i\leq m_n$.
By compactness of $\Lambda$ we can suppose that $x_n\to x\in \Lambda$, $y_n\to y\in M$. Then we have $diam(N^r_{\de}(x_n))>\eta>0$ for any $x_n$. Since $X$ is a $C^1$-vector field,  the normal direction of $X$ varies continuously with $x$, so we have that  $y\in N^r_{\de}(x)$. But now, the continuity of the holonomy maps implies that $$d(P_{x,it}(x),P_{x,it}(y))\leq \de||X(P_{x,it}(x)||,$$ for every $i\in \Z$ and then $x=y$, a contradiction, since $d(x,y)>\eta$.
\end{proof}

\begin{rmk}
In \cite{KS} it is proved that any expansive flow is uniformly expansive. So in the previous result, we have $\phi|_\Lambda$ is expansive, and therefore, it is uniformly expansive by the non-singularity of $\Lambda$. But with this restriction, we only obtain a uniform time of separation between the orbits in $\Lambda$. On the other hand, our result gives a uniform time in which $\Lambda$ expels any point $x$ in a neighborhood of $\Lambda$, if $x$  is not too close to $\Lambda$. 
\end{rmk}
Next, we use the uniform R-expansiveness to obtain some "hyperbolic behavior" for the  R-stable (R-unstable) sets of points away from singularities. i.e., these sets need to contract uniformly in the future (in the past).

\begin{proposition}[Uniform contraction]\label{UnifContr}
For any $0<\eta<\de A_\Lambda$ and any $t>0$,  there is $J>0$ such that $$P_{x,nt}(S_{\de}(t,x))\subset S_{\eta}(t,P_{x,nt}(x)) \text{ and }   P_{x,-nt}(U_{\de}(t,x))\subset U_{\eta}(t,P_{x,nt}(x))$$   for every $n\geq J$ and every $x\in \Lambda$.
\end{proposition}

\begin{proof}
Let us fix  $0<\eta<\de A_\Lambda$ and $t>0$. Let $J$ be given by the previous theorem. suppose there exists $x\in \Lambda$ such that   $$P_{x,nt}(S_{\de}(t,x))\not\subset S_{\ga}(t,P_{x,nt}(x)).$$ Then there is some $y\in S_{\de}(t,x))$ and $n\geq J$ satisfying  $$d(P_{x,nt}(x),P_{x,nt}(y))>\eta.$$ By the choice of $J$ we must have $$d(P_{x,(n+i)t}(x),P_{x,(n+i)t}(y))>\de||X(P_{x,(n+i)t}(x))||$$ for some $-J\leq i\leq J$. On the other hand, since $n\geq J$, we have $n+i>0$. But $y\in \mathcal{S}_{\delta}(t,x)$, and therefore  $$d(P_{x,(n+i)t}(x),P_{x,(n+i)t}(y))\leq\de||X(P_{x,(n+i)t}(x))||,$$
a contradiction. 
\end{proof}

The following result is an easy corollary of Proposition \ref{UnifContr}.

\begin{proposition}\label{omega}
If for some $t>0$, we have $y\in S_{\eps}(t,x)$, then $\omega(x)=\omega(y)$.
\end{proposition}

 Before proving the proposition, we can make some remarks that will be used in the next results. By Proposition \ref{UnifContr}, if $y\in S_{\eps}(t,x)$, then $d(P_{x,nt}(x),P_{x,nt}(y))\to 0$ as $n\to \infty$. Arguing as in the proof of Theorem \ref{charac}, one can construct a reparametrization $\rho$ such that  $$\lim\limits_{t\to\infty}d(\phi_t(x),\phi_{\rho(t)}(y))= 0.$$

 \begin{proof}
  Let $z\in \omega(x)$ and suppose that $y\in S_{\eps}(t,x)$. If $t_k\to \infty$ is such that $\phi_{t_k}(x)\to z $, then the previous remarks implies  $\phi_{\rho(t_k)}(y)\to z$ and therefore $z\in \omega(y)$. The reverse inclusion is analogous.
\end{proof}

\begin{corollary}
Let $x$ be a periodic point with period $\pi(x)=t$. For every $0<\eta\leq A_{O(x)}$ there exists $J$ such that:  $$P_{x,nt}(S_{\de}(t,x))\subset S_{\eta}(t,P_{x,nt}(x)) \text{ and }   P_{x,-nt}(U_{\de}(t,x))\subset U_{\eta}(t,P_{x,-nt}(x)),$$   for every $n\geq J$.

\end{corollary}

\subsection{Rescaled-stable points}
Let us now introduce the concept of R-stable and R-unstable points. We want to mention that the results here are inspired by the techniques developed in \cite{Her}.
Define the positive and negative $n$-$\eps$-R-dynamical ball
centered at $x$, respectively by:

$$N^r_t(x,n,\eps)=\{y\in N^r_{\eps}(x); d(P_{x,it}(x),P_{x,it}(y))\leq \eps||X(P_{x,it}(x))||, 0\leq i\leq n\}.$$

\begin{center}
    and
\end{center}

$$N^r_{-t}(x,n,\eps)=\{y\in N^r_{\eps}(x); d(P_{x,it}(x),P_{x,it}(y))\leq \eps||X(P_{x,it}(x))||, -n\leq i\leq 0\}.$$

\begin{definition}

We say that $x\in M\setminus \sing(\phi)$ is an R-stable  (R-unstable ) point of $\phi$ if for every  $t>0$, the set $\{S_{\eps}(t,x)\}_{\eps>0}$ ($\{U_{\eps}(t,x)\}_{\eps>0}$) is a neighborhood basis for $x$ on $N^r_{\de}(x)$. In other words, if for every $\eps>0$, there is some $\eta>0$ such that if $y\in N^r_{\eta}(x)$ and $d(x,y)\leq \frac{\eta}{L^t}||X(x)||$, then $$d(P_{x,nt}(x),P_{x,nt}(y))\leq \eps||X(P_{x,nt}(x))||$$ for every $n\geq 0$ ($n\leq 0$).
\end{definition}

The following theorem is a trivial consequence of the definitions.

\begin{theorem}\label{Rcharac}
If $\overline{O(x)}\cap \sing(\phi)=\emptyset$, then are equivalent:

\begin{enumerate}
    \item $x$ is a R-stable point.

    \item $S_{\de}(t,x)$ is a neighborhood of $x$ on $N^r_{\de}(x)$.

    \item For every $t>0$ there is  $0<\eps_0<\de$, such that for every $0<\eps\leq \eps_0$, there is  $J>0$ such that $$N^r_{t}(x,J,\eps)= S_{\eps}(t,x),$$
\end{enumerate}

\end{theorem}
\begin{proof}
The proof of the implication $(1)\Rightarrow (2)$ is evident from the definition of R-stable points. 

To prove $(2)\Rightarrow(3)$, fix $t>0$ and let $\eps_0>0$ be such that  $N^r_{\eps_0}(x)\subset S_{\de}(t,x)$. Fix $0<\eps\leq \eps_0$ and let $J$ be given by Theorem \ref{UnifContr}, with respect to $\eps$.  By definition, we have $\mathcal{S}_{\eps}(t,x)\subset N^r_{t}(x,J,\eps)$.  Now, notice that $N^r_{\eps}(x,J,\eps)\subset N^r_{\eps_0}(x)$. Then Theorem \ref{UnifContr} implies $$P_{x,Jt}(N^r_t(x,J,\eps)\subset S_{\eps}(t,P_{x,Jt}(x)),$$  and therefore $N^r_t(x,J,\eps))\subset S_{\eps}(t,x)$.

To see why $(3)\Rightarrow (1)$, just notice that, by continuity, $N^r_{t}(x,J,\eps)$ is a neighborhood of $x$ in $N^r_{\eps}(x)$.
\end{proof}

\begin{remark}
    An equivalent result clearly holds for $R$-unstable points.
\end{remark}

Hereafter we will always suppose $x\in \Lambda$, where $\Lambda$ is a compact invariant set without singularities.
Before stating our next result, we would like to state a result from \cite{JNY}, which will be used in our next proof. Let us denote $C_{\phi}(M)
$ for the set of non-negative functions $f:M\to [0,\infty)$ such that $f(x)=0$ if, and only if $x\in \sing(\phi)$.  Note that for any $\de>0$, the functions $\de||X(x)||$ belongs to $C_{\phi}(M)$. The next result will help us to find continuity properties for the R-holonomy maps.

\begin{lemma}[\cite{JNY}]\label{ContRH} Let $\phi$ be a continuous flow on $M$.
\begin{enumerate}
\item For any $e\in C_{\phi}(M)$ and $T>0$ we can find $r\in C_{\phi}(M)$ such that:\\ if $d(x,y)\leq r(x)$, then $$d(\phi_{t}(x),\phi_{t}(y))\leq e(\phi_t(x)),$$ for every $t\in [-T,T]$
\item For any $e\in C_{\phi}(M)$ there is some $r\in C_{\phi}(M)$ such that $$r(x)\leq\max\{e(y);y\in B_{r(x)}(x) \}.$$
\end{enumerate}
\end{lemma}

\begin{proposition}\label{Sperio}
If $x\in \Lambda$  is R-stable and recurrent, then $x$  is periodic.
\end{proposition}

\begin{proof}
Suppose that $x\in \Lambda$ is recurrent and R-stable. Fix $\eta>0$ such that $N^r_{\eta}(x)\subset S_{\eps}(t,x)$. Since $x$ is recurrent we can find a sequence $t_k\to \infty $ such that $\phi_{t_k}(x)\to x$. By Theorem \ref{RFB}, after by possibly reduing $\eta$,  we have $B^{r}_{\frac{\eta}{3}}(x)$  is contained on the R-flow box of $N^{r}_{\eta}(x)$. Notice that by euclidean algorithm we can write every $t_k$  in the form $$t_k=n_kt+r_k, $$
where $0\leq r_k<t$. 

 If we choose  $t_k$ big enoguh, we can assume $\phi_{t_k}(x)\in B^r_{\frac{\eta}{3}}(x)$. Due to Theorem \ref{UnifContr}, after possibly enlarging $t_k$,  we can also assume $diam(P_{x,n_kt}(N^r_{\eta}(x))$ is very small. By the continuity of $\phi$, for any $y\in P_{x,n_kt}(N^r_{\eta}(x))$, there is $r^y_k$ close to $r_k$ such that $$\phi_{r^y_k}(y)\in N^r_{\eta}(x).$$ Therefore, we can define a projection  $$\pi: P_{x,n_kt} (N^r_{\eta}(x))\to N^r_{\eta}(x)$$ by setting  $\pi(y)=\phi_{r^y_k}(y)$. Now, notice that $$F=\pi\circ P_{x,n_kt}:N^r_{\eta}(x)\to N^r_{\eta}(x)$$ is a $C^1$  map. Thus, we can use the Brower fixed point theorem to find a fixed point $z\in N^r_{\eta}(x)$ for $F$. Consequently, $z$ is a periodic point for $\phi$.  Finally,  by the previous proposition, we have that $x\in \omega(x)=\omega(z)=O(z)$ and this finishes the proof.
\end{proof}

In the next results, we will see that the R-stable points of a non-singular subset $\Lambda$ of an R-expansive flow are formed by periodic orbits that are isolated from $\Lambda$.

\begin{proposition}\label{R-Open}
Suppose $\phi$ is R-expansive and let $x\in M$ be such that $\overline{O(x)}\cap \sing(\phi)=\emptyset$. If $x$ is a R-stable point, there is a neighborhood of $x$ on $N^r_{\de}(x)$ formed by R-stable points.
\end{proposition}

\begin{proof}

Suppose that $x$ is a R-stable point and fix $0< 4\eps<\de$ such that

 $$\left(\bigcup_{t\geq 0}\overline{N^r_{\eps}(\phi_t(x)})\right)\cap \sing(\phi)=\emptyset$$

Since $x$ is R-stable, then there is some $0<\eta<\eps$ such that $N^r_{\eta}(x)\subset S_{\eps}(t,x)$. This implies that there is $A>0$ such that if $y\in N^r_{\eta}(x)$, then  $\inf\limits_{t\geq 0}\{||X(\phi_t(y))||>A>0$. Now fix $\nu>0$ and set $0<\ga\leq \nu A$. Fix  $y\in N^r_{\eta}(x)$. Proposition \ref{UnifContr} combined with Theorem \ref{RFB} implies that we can find $T>0$ such that  $$d(P_{y,nt}(y),P_{y,nt}(z))\leq \frac{\ga}{L^t}$$ for any $z\in B_{\eta}(x)$ and for any $n\geq N$. Finally, the continuity of the holonomy maps allows us to find $\mu>0$ (Lemma \ref{ContRH}) such that if $z\in N^r_{\eta}(x)$ and $d(z,y)<\mu$, then $$d(P_{y,nt}(y),P_{y,nt}(z))\leq \frac{\ga}{L^t}$$ for $0\leq n\leq N_{\eta}$.  But this implies $N^r_{\mu}(y)\subset S_{\nu}(y,t)$ and therefore, $y$ is R-stable.
\end{proof}

\begin{lemma}\label{vizest}
Let $\Lambda$ be a non-singular set and fix $t>0$. Suppose $x\in \Lambda$ is an R-stable point. There exist $\rho>0$ and $T>0$ such that $$N^r_{\rho}(P_{x,-nt}(x))\subset S_{\frac{\de}{3}}(t,P_{x,-nt}(x)),$$ for every $n\geq N$.    
\end{lemma}

\begin{proof}
    Let $\eps>0$ be such that $N^r_{\eps}(x)\subset S_{\frac{\de}{3}}(t,x)$ and let $T_{\eps}$ be given by Theorem \ref{uef}.  By continuity, we can find  $\rho>0$ such that $N^r_{\rho}(y)\in N^r_{-t}(y,T_{\eps},\eps)$, for any $y\in \Lambda$. We claim that the Lemma holds for $T=2T_{\eps}$. Indeed, first notice that $$N^r_{t}(P_{x,-nt}(x),n,\eps)\subset S_{\frac{\de}{3}}(t,P_{x,-nt}(x)),$$
     for every $n\geq0$.
     If the lemma does not hold, there should be $n\geq T$  and $$y\in N^r_{\rho}(P_{x,-nt}(x))\setminus N^r_{t}(P_{x,-nt}(x),n,\eps).$$

     In particular we can find $$z\in N^r_{\rho}(P_{x,-nt}(x))\cap \partial N^r_{t}(P_{x,-nt}(x),n,\eps)$$
     Therefore there is $k> T_{\eps}$ such that $$d(P_{x,-(n+k)t}(x),P_{x,-(n+k)t}(z))\geq\eps A_{\Lambda}.$$
     
     But now, Theorem \ref{uef} implies 

     $$\max\limits_{|j|\leq T_{\eps}}\{d(P_{x,(-n+k+j)t}(x),P_{x,-(-n+k+j)t}(z))\}>\frac{\de}{3}A_{\Lambda}  $$
     contradicting the choice of $z$.
\end{proof}

\begin{theorem}\label{Stable}
Let $\phi$ be an R-expansive flow and $\Lambda\subset M$ be a compact invariant set without singularities. If $x\in \Lambda$ is an R-stable or R-unstable point, then $x$ is periodic.

\end{theorem}

\begin{proof}

We will only present a proof for R-stable points, since the case of R-unstable points is totally analogous.  Suppose $x\in \Lambda$ is an R-stable point, let $\de>0$ be an R-expansiveness constant of $\phi$ and fix $t>0$. By the compactness of $\Lambda$, we can find a sequence $n_k\to \infty$ such that $$\lim\limits_{k\to\infty}P_{x,-n_kt}(x)= z.$$  In particular, $z$ is a regular point. Let  $0<\gamma<\de$ be such that $N^r_{\gamma}(y)$ is well defined and it is a cross-section of time $t$, for every $y\in \Lambda$. Let $\rho>0$ and $T>0$  be given by Lemma \ref{vizest}.  For every $n_k\geq T$, denote $$V_k=\phi_{[-\xi,\xi]}(N^r_{\rho}(P_{x,-n_k t}(x)).$$ 
Notice that $z\in int(V_k)$, if $k$ is big enough.
\\
\par \textit{\textbf{Claim:} $z$ is a R-stable point.}
\\

To see why the claim holds, we first fix $n_k\geq T$.
 Since $z\in V_k$, one can find $-\xi<s_z<\xi$  such that $$\phi_{s_z}(z)\in N^r_{\rho}(P_{x,-n_kt}(x)).$$ 
 By  Proposition \ref{vizest}, $\phi_{s_z}(z)$ is a R-stable point. Now, Proposition \ref{omega} implies $$\omega(P_{x,n+kt}(x))=\omega(x)=\omega(z).$$ Moreover, there is a reparametrization $\rho_z$ such that $$\lim\limits_{s\to\infty}d(\phi_s(P_{x,n_kt}(x)), \phi_{\rho_z(s)}(\phi_{s_z})(z))=0.$$ 
But this implies $z\in \omega(z)$ and therefore $z$ is recurrent. By Theorem \ref{Sperio}, we obtain that $z$ is periodic, and as a consequence, $x$ is also periodic.

\end{proof}

\begin{corollary}\label{isolated}
 If $x\in \Lambda$ is a R-stable or R-unstable point, then $O(x)$ is isolated from $\Lambda$. In particular, $\Lambda$ contains at most a finite number of orbits of $R$-stable or R-unstable points.  
\end{corollary}

\begin{proof}
Let $x\in \Lambda$ be a $R$-stable point. By Theorem \ref{Stable} $O(x)$ is a compact set. Suppose $O(x)$ is not isolated from $\Lambda$, then there is a sequence o points $x_n\in \Lambda\setminus O(x)$ such that $x_n\to x$. By combining Theorems \ref{R-Open}  and \ref{Stable} we obtain that  $x_n$ is periodic for $n$ sufficiently large, contradicting the expansiveness of $\phi_t|_{\Lambda}$.   
    
\end{proof}

\subsection{Existence of non-trivial stable sets}
If $A\subset X$ and $x\in A$, we denote by $C(A,x)$ the connected component of $A$ containing $x$. For any $x\in M$, $t>0$ and $\eps>0$, we denote  $$CS_{\eps}(t,x)=C(S_{\eps}(t,x),x) \textrm{  and } CU_{\eps}(t,x)=C(U_{\eps}(t,x),x).$$

Our first result deals with the existence of connected pieces of local R-stable and R-unstable sets
with large diameter.
This problem was first solved for expansive homeomorphisms of surfaces by J. Lewowicz and K. Hiraide, who independently classified such systems (they are conjugate to pseudo-Anosov diffeomorphisms).
This was later generalized for Bowen-Walters expansive flows in \cite{KS}.
Let us  denote $\Gamma^r_{\ga}(p)$ for the sphere of radius $\ga||X(p)||$ centered at $p$.

\begin{theorem}\label{ConStab}
If $\phi$ is an R-expansive flow,
$\Lambda\subset M$ is a compact invariant set without singular points and $\delta>0$
then there is $\eta>0$ such that
$$CS_{\de}(t,p)\cap \Gamma^r_{\eta}(p)\neq \emptyset \textrm{ and }  CU_{\de}(t,p)\cap \Gamma^r_{\eta}(p)\neq \emptyset.$$
for all $t>0$ and any point $p\in \Lambda$ which is not $R$-stable or $R$-unstable.
\end{theorem}

\begin{proof}
The proof is based on the following claim:

\textit{Claim: For every $0<\eps<\de$, and $\eta>0$, there is some $K=K_{\eps,\eta}$ such that $$N^r_{\eta}(x)\not\subset N^r_t(x,K,\eps)\textrm{ and } N^r_{\eta}(x)\not\subset N^r_{-t}(x,K,\eps)$$
for every $x\in \Lambda$.}

If the claim is false, we can find $\eps>0$ and $\eta>0$ and a sequence of points $x_k\in \Lambda$ such that $N^r_{\eta}(x_k)\subset N^r_t(x_k,k,\eps)$ for any $k>0$. Now, if we suppose that $x_k\to x$, then $x$ must an R-stable point of $\Lambda$ and this is a contradiction. The case of R-unstable points is analogous and the claim is proved.

Now fix $x\in \Lambda$, $0<\eps<\de$ and  let $T>0$ be given by Theorem \ref{uef}. Let $\eta>0$ be such that if $d(x,y)\leq \eta$, then $d(P_{x,nt}(x),P_{x,nt}(y))\leq \eps$ for $|n|=0,..., T$. Fix some $n\geq \max\{T,K_{\eps,\eta}\}$. By the claiming, we have that $$P_{x,-nt}(N^r_{\eta}(P_{x,nt}(x)))\not\subset C(N^r_t(x,n,\eps),x).$$ Thus there is some $$y_0\in P_{x,-nt}(N^r_{\eta}(P_{x,nt}(x))\cap \partial C(N^r_t(x,n,\eps),x).$$
In particular, this implies that for some $0\leq k\leq n$, we have that $$d(P_{x,kt}(x),P_{x,kt}(y_0))=\eps.$$

But now, $k\notin [n-T,n-1]$, by the choice of $\eta.$ Also $k\notin [T,n-T]$, otherwise there should exists some $0\leq j\leq n$ such that $d(P_{x,jt}(x),P_{x,jt}(y_0))> \de$ contradicting $y_0\in N^{r}_{t}(x,n,\eps)$. Thus, $k\in [0, T]$ and therefore $d(x,y_0)>\eta$, by the choice of $\eta$.

Finally, we have that for any $n\geq \max\{T,K_{\eps,\eta}\}$ we have that $C(N^{r}_t(x,n,\eps),x)$ is a connected set with diameter greater than $\eta$. Thus by the compactness of the continuum hyperspace of $M$ we have that the set $$\bigcap_{n>0}\overline{C(N^r_t(x,n,\eps),x)}$$ is connected set contained in $S_{\eps}(t,x)$ with diameter greater than $\eta$. Since the case for the R-unstable sets is analogous, the theorem is proved.
\end{proof}

\begin{remark}
    In the previous result, we are assuming $\Lambda$ is non-singular. Thus, $\phi|_{\Lambda}$ is $BW$-expansive. This, combined with the results in \cite{ACP}, could lead us to wonder whether the existence of a non-trivial $BW$-expansive subset $\Lambda$ implies the existence of non-trivial connected components of stable or unstable sets for points in $\Lambda$, concerning the flow $\phi|_{\Lambda}$. This is not true. Indeed, notice that if $p$ is periodic, then $\phi|_{O(p)}$ is clearly $BW$-expansive, by in this case, the connected  component of  $CS_{\delta}(p)$  and  $CU_{\delta}(p)$ are $\{p\}$. This is due to the lack of enough dimension on $\phi|_{O(p)}$ to reproduce the arguments in \cite{ACP}. This illustrates that, although in our result $\phi|_{\Lambda}$ is $BW$-expansive, we need the global $R$-expansiveness to derive the conclusion. 
\end{remark}

\section{The Topological Entropy of R-expansive flows and Some Examples}\label{entropysection}

In this section, we will explore the topological entropy of R-expansive flows. We will divide this section into two parts. The former deals with general R-expansive flows containing Lyapunov stable sets, while in the second part, we will present some examples of R-expansive flows to which our results can be applied or not.

\subsection{R-expansive flows with non-singular Lyapunov stable sets.}

In the next result, we derive some conditions to obtain positive topological entropy for R-expansive flows.

\begin{theorem}\label{LyapEntr}

    Let $\phi$ be a R-expansive flow. If there exists a non-singular Lyapunov stable set  $\Lambda\subset M$, such that $\Lambda$ is not a finite union of compact orbits, then $h(\phi)>0$.
\end{theorem}

\begin{proof}
Let $\Lambda$ be a non-singular set as in the hypothesis. Let $\Lambda'$ be the set formed by the orbits of all $R$-stable points and the orbits of all $R$-unstable points of $\Lambda$. Since $\Lambda$ is not a finite union of periodic orbits, $\Lambda_0=(\Lambda\setminus \Lambda')\neq \emptyset$. Moreover, Corollary \ref{isolated} implies $\Lambda_0$ is compact, invariant, and without $R$-stable or $R$-unstable points. In addition, the Lyapunov stability of $\Lambda$ implies that $\Lambda_0$ is also Lyapunov stable. Now, Theorem \ref{Stable} implies that there exists $\eta>0$ such that $CU_{\eta}(t,x)$ is non-trivial for any $x\in \Lambda_0$. Let us now  fix some constants:

\begin{enumerate}
\item Fix $\de>0$  the constant of R-expansiveness of $\phi$.

\item Fix $0<\eps\leq\de$ such that $\overline{B_{\eps}(\Lambda_0)}\cap \sing(\phi)=\emptyset$.

\item Let  $\ga>0$ be given by the Lyapunov stability of $\Lambda_0$ with respect to $\eps$.
\item fix $x\in \Lambda_0$ and let $y\in CU_{\eta}(t,x)$ be such that $y\neq x$.
\end{enumerate}

Proposition \ref{UnifContr} implies that $d(P_{x,-nt}(y),P_{x,-nt}(x))\to 0$, while Theorem \ref{RFB} implies that in fact $d(O(x),\phi_{-t}(y))\to 0$. On the other hand, the Lyapunov stability of $\Lambda_0$  guarantees that $\phi_t(y)\in B_{\eps}(\Lambda_0)$ for any $t\geq0$ (see Figure 1).

Last facts  imply that $$\Lambda_1=\overline{\bigcup_{x\in \Lambda_0}\bigcup_{t\in \R}\phi_t({\overline{CU_{\eta}(t,x)}})}$$
is a compact and invariant set contained in $\overline{B_{\eps}(\Lambda_0)}$ and hence non-singular. In particular, it is expansive and has dimension greater than one, since it contains $O(x)$ and $U_{\eta}(t,x)$. So, we conclude by Theorem \ref{EntCW} that $h(\phi)>0$.

\begin{figure}[h]
  \centering
 \includegraphics[scale=0.45]{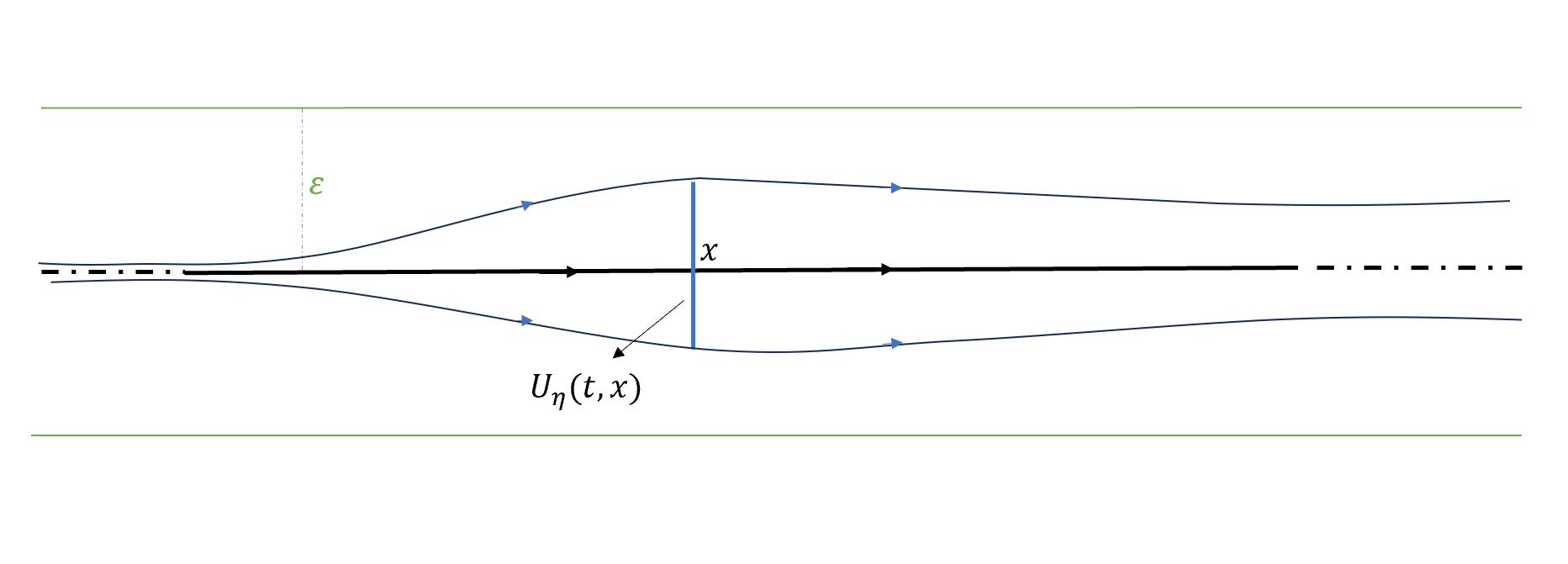}
    \caption{The idea behind the proof of Theorem \ref{LyapEntr}}
   \label{fig:my_label}
\end{figure}
\end{proof}

In contrast with Theorem \ref{EntCW},  in the previous result, we are not assuming any dimensional hypothesis on $\Lambda$. Instead, our techniques allow us to construct a compact and invariant with enough dimension, even if we begin with a one-dimensional set $\Lambda$.  In addition, we obtain as an immediate consequence of Theorems \ref{LyapEntr} and \ref{artiguerelation} we obtain the following result:

\begin{corollary}\label{EntHip}
    Let $\phi$ be a $k^*$-expansive and $\Lambda$ be a non-singular Lyapunov stable subset of $M$. If  $\Lambda $ is not a finite union of periodic orbits, then $h(\phi)>0$.
\end{corollary}

\subsection{Examples}\label{examples}

We end this work by presenting some new examples of R-expansive flows and showing their relation with the results we obtained. Our first example illustrates that although $k^*$-epansiveness implies R-expansiveness, there are examples of non-trivial R-expansive flows that are far from being $k^*$-expansive. Furthermore, our results do not apply to this example.

\begin{example}\label{ExR2}
 Consider $M=\mathbb{T}^3=S^1\times\mathbb{T}^2$. We begin by defining a periodic flow $\psi$ on $M$ induced by a vector field with velocity constant and equal to one. Here we will see  $M$ as the product $[-2,2]\times \T^2$, where the end points of $[-2,2]$ are identified.  Let us consider on $M$ the vector field $X$ constant and equal to $(1,0,0)$. Thus $X$ generates the flow $\psi$ desired.

Now we modify this flow to obtain an R-expansive flow. First consider a smooth non-negative function $\rho$ on $M$ satisfying the following conditions:

\begin{enumerate}

\item $\rho$ is  constant along the fibers $\{x\}\times \T^2$.
\item $\rho((x,y,z))=1$, if $(x,y,z)\in [-2,-1]\times \T^2$ or $(x,y,z)\in [1,2]\times \T^2$. 
\item $\rho((x,y,z))=0$ if, and only if, $(x,y,z)\in \{0\}\times \T^2$.
\item $\rho((x,y,z))$ decreases in $[-1,1]\times \T^2$, as $x\to 0$.

 \end{enumerate}

 \begin{figure}[h]
     \centering
   \includegraphics[scale=0.58]{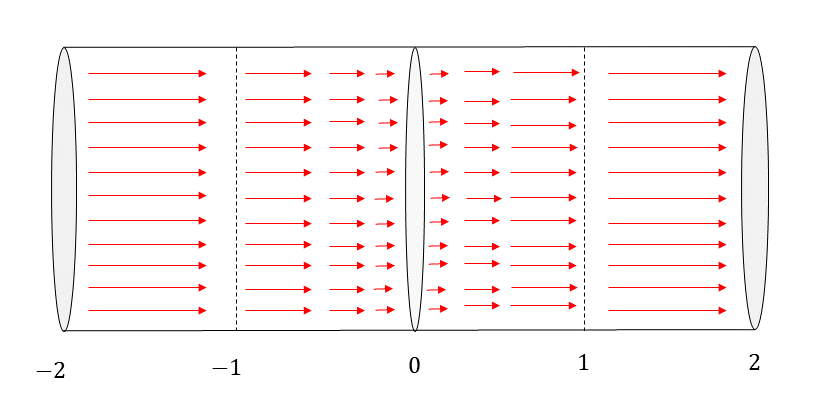}
   \caption{Chain-transitive R-expansive flow on $\mathbb{T}^3$ with zero topological entropy.}
    \label{fig:my_label}
 \end{figure}

Let $\phi$ be the flow generated by the field $\rho X$ (see  Figure 3).

\vspace{0.1in}
\textit{Claim: $\phi$ is R-expansive}

\vspace{0.1in}

To prove the claim, we proceed as follows. Fix some regular point $p=(x,y,z)\in M$. Notice that $\{x\}\times \T^2$ is a cross-section through $p$ for any time $t>0$. So fix some $t>0$ and $\de>0$. The set $S_{\de}(t,x)(p)$ is formed by all the points $q\in N^r_{\de}(p)$ such that
$$d(P_{p,nt}(p),P_{p,nt}(q))\leq ||X(P_{p,nt}(p))||$$
for any $n>0$. But by the choice of $\rho$, one  $||X(P_{p,nt}(p))||\to 0$ as $n\to \infty$.
On the other hand, $\phi$ acts isometrically on the fibers $\{x\}\times \T^2$ and this implies that $$d(P_{p,nt}(p),P_{p,nt}(q))=d(p,q)$$
for any $q\in N^r_{\de}(p)$ and every $n\in \Z$. Thus we have that $S_{\de}(t,x)(p)=\{p\}$. A similar argument shows that $U_{\de}(t,x)(p)=\{p\}$. This proves that $\phi$ is R-expansive. In addition, $h(\phi)=0$, since the non-wandering set of $\phi$ is formed by fixed points. Finally, notice that our results do not apply to this example, since $\overline{O(X)}\cap Sing(X)\neq\emptyset$, for every $x\in M$.
\end{example}

\begin{remark}
    In \cite{P}, it is proved that $BW$-expansive flows cannot exist on three-dimensional manifolds with a fundamental group of sub-exponential growing. On the other hand, in \cite{Ar4}, it is proved that the same does not hold for $k^*$-expansive flows by exhibiting an example of such a flow on $S^3$. The previous example also shows that the same does not hold for R-expansive flows which are not $k^*$-expansive.
\end{remark}

The next is an example to illustrate our main results.

\begin{example}\label{double-rovella}
To begin with, let $A\subset \R^3$ be the geometric Rovella's attractor defined as in \cite{CSM}. For a detailed discussion on the construction and properties of the Rovella attractor, we refer the reader to \cite{Ro} and \cite{MM}. By following techniques analogous to the techniques used in \cite{Ar4}, one can obtain $A$ by a flow $X$ on  $\mathbb{R}^3$ with the following properties:

 \begin{enumerate}

 \item There exists a solid two-torus $S\subset \R^3$ such that the vector field $X$ is transversal to $\partial S$ and points inwardly $S$.
 \item $A\subset int(S)$.
 \item $A=\cap_{t\geq0} \phi'_t(S)$, where $\phi'$ is the flow generated by $X$.

 \end{enumerate}
  So, by considering $-X$, we obtain a Rovella's repeller $R$  delimited by a two-torus $S$ and whose $-X$ is transversal to $S$ and points outwardly $S$.

  Next, we construct another two-torus $S'$ as follows. First, consider the vector field $Y$ obtained in Subsection 1.2.3  of \cite{Araujo} on the solid torus $T$. The vector field $Y$ is inwardly transversal to $\partial T$ and generates a flow $\psi$ such that $$\bigcap_{t\in\R}\psi_t(T)=\Gamma\cup O,$$
  where $\Gamma$ is a suspension of a Plykin attractor and $O$ is a hyperbolic source. 

  Let $T'$ be a copy of $T$ and $Y'$ be a copy of $Y$ over $T'$. Let $S'$ be the solid two-torus obtained by connecting $T$ and $T'$ along a cylinder $C$ as in Figure 4.

   \begin{figure}[h]\label{two-toro}
     \centering
   \includegraphics[scale=0.58]{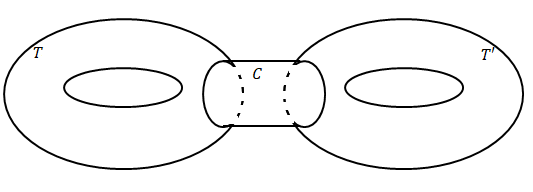}
   \caption{Construction of $S'$.}
    \label{fig:my_label}
 \end{figure}

  Now, we want to definite a vector $X'$ on $S'$ which is inwardly transverse to $\partial S'$ and, when restricted to $int(T)$ and $int(T')$, equals $Y$. This construction can be coherent by adding a hyperbolic saddle singularity on $C$ as in Figure 5.

   \begin{figure}[h]\label{FlowinC}
     \centering
   \includegraphics[scale=0.58]{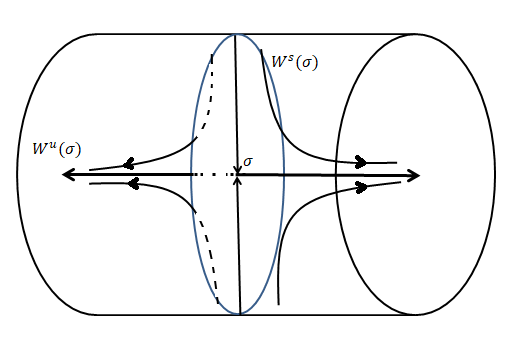}
   \caption{The hyperbolic saddle in $C$.}
    \label{fig:my_label}
 \end{figure}

  By the theory of collars (see \cite{Koz}), o,ne can smoothly glue $S$ and $S'$ along their boundaries by a diffeomorphism $\rho:\partial S\to \partial S'$.  
  This generates a smooth, compact, and boundaryless manifold $M$. From now on, we regard $\partial S$ and a submanifold of $M$.
  
  Moreover, since $-X$ is outwardly transverse to $\partial S$ and $X'$ is inwardly transverse to $\partial S$, the collar theory allows us to obtain a  small  neighborhood $U$ of $\partial S$ and a $C^1$-vector field $V'$ over $M$ such that $V'|_{S\setminus U}=-X$, $V'|_{S'\setminus U}=X'$ and $V'|_{U}$ induces a tubular flow through $\partial S$.  
In this way, the flow induced by $V'$ has the following properties: 
\begin{enumerate}
    \item $M$ contains two hyperbolic attractors $\Gamma_1$ and $\Gamma_2$ which are suspensions of the Plykin attractor. 
    \item $M$ contains a Rovella's repeller $\Lambda$.
    \item $M$ contains a pair of hyperbolic sources $O_1$ and $O_2$.
    \item $M$ contains a hyperbolic saddle singularity $\sigma$ disjoint from $\Lambda\cup \Gamma_1\cup \Gamma_2\cup O_1\cup O_2$.
    \item The orbit of any point in $M$ which is not contained $\Omega(V')\cup W^u(\sigma)$ transversely crosses $\partial S$.
\end{enumerate}

Notice that in order to apply our main results, we need to ensure that $V'$ is R-expansive. For this sake, we modify $V$ similarly to the Example \ref{ExR2}.

  Let $U'\subset U$ be a neighborhood of $\partial S$  and let $\rho:M\to R$ be a $C^{\infty}$-function such that:
 \begin{itemize}
 \item $\rho(x)\geq 0$, for every $x\in M$.
  \item $\rho(x)=1$, for every $x\in M\setminus U'$.
  \item $\rho(x)=0$ if, and only if $x\in \partial S$.
\end{itemize}
Next consider $V=\rho V'$ and let $\phi_t$ be the flow induced by $V$. Notice that now we have $$\Omega(V)=\Gamma_1\cup \Gamma_2\cup \Lambda\cup O_1\cup O_2\cup \sigma\cup \partial S,$$
where all the above unions are pairwise disjoint.

Notice that  if $x$ is  wandering point, then $x$ satisfies only one of the following behaviors.
\begin{enumerate}
    \item $x\in W^u(\sigma)$.
    \item $\phi_t(x)\to \Gamma_1\cup \Gamma_2\cup \sigma$ as $t\to \infty$ and there is some $p\in \partial S$ such that  $\phi_{t}(x)\to p$, as $t\to -\infty$.
    \item There is some $p\in \partial S$ such that $\phi_t(x)\to p$, as $t\to \infty$ and $\phi_{t}(x)\to \Lambda\cup O_1\cup_2\cup \sigma$, as $t\to -\infty$.
\end{enumerate}

\begin{claim}$\phi_t$ is  R-expansive.
\end{claim}
\begin{proof} 
Note that the R-expansiveness of $\Omega(Y)$ is immediate. Indeed, $\Gamma_1$ and $\Gamma_2$ are hyperbolic and, therefore, expansive. Since $\Lambda$ is $k^*$-expansive, then it is $R$-expansive and finally, $O_1,O_2,\sigma$ and $\partial S$ are trivially $R$-expansive. Next, we divide the proof by cases:

\textit{Case 1:} If $x$ is wandering and $y\in \Omega(V)$, then $x$ and $y$ are separated by conditions $(2)$ and $(3)$ above. 

\textit{Case 2:} If $x$ and $y$ are wandering points and $x$ satisfy condition $(2)$ and $y$ satisfy condition $(3)$, then $x$ and $y$ are also trivially separated.

\textit{Case 3:} If both $x$ and $y$ satisfy condition $(2)$ then $\phi_t(x)\to p_x$ and $\phi_t(y)\to p_y$ as $t\to -\infty$. Since, $V(\phi_t(x))\to 0$ as $t\to -\infty$, then $$d(\phi_t(x),\phi_{h(t)}(y))<\delta |V(\phi_t(x))|,$$ for every $t\in \R$, if and only if, $p_x=p_y$. But this implies $y\in O(x)$.

\textit{Case 4:} If  $x\in W^u(\sigma)$ and $y\notin W^u(\sigma)$ then $x$ and $y$ are separated by $\phi_t$.

\textit{Case 5:} If both $x$ and $y$ are contained in $W^u(\sigma)$, but in different components of $W^u(\sigma)\setminus \{\sigma\}$, then $x$ and $y$ are separated in the future by $\phi_t$.

\textit{Case 6:} If both $x$ and $y$ are contained in the same component of $W^u(\sigma)\setminus \{\sigma\}$, then $x$ and $y$ are separated in the future by $\phi_t$, unless they $y\in \phi_{[-\eps,\eps]}(x)$, for small $\eps$ (see \cite{BW}).

\end{proof}

 Finally, $\Gamma_1$ and $\Gamma_2$ Lyapunov stable sets under the hypothesis of Theorem \ref{LyapEntr}. Thus, we can to obtain $h(\phi)>0$.

\end{example}

\textit{\textbf{Acknowledgements:} The authors are thankful to the anonymous referees for their valuable comments that helped us to improve this manuscript}


\begin{thebibliography}{00}%\geq

\bibitem{AH} N. Aoki and K. Hiraide, \emph{Topological theory of dynamical systems: Recent advances.} North-Holland Mathematical Library, vol. 52, North-Holland Publishing Co., Amsterdam, 1994.
 
 \bibitem{Araujo} V. Ara\'{u}jo, \emph{On the number of ergodic physical/SRB measures of singular-hyperbolic attracting sets.} Journal of Differential Equations,
354, 373-402, (2023).

\bibitem{APPV} V. Ara\'{u}jo,  E.R. Pujals, M.J. Pacifico, M. Viana  \emph{Singular-hyperbolic attractors are chaotic.} Trans. A.M.S. 361, 2431-2485 (2009)

\bibitem{ACP} A. Arbieto, W. Cordeiro, M.J. Pacifico,  \emph{Continuum-wise expansivity and entropy for flows.} Ergod. Theory Dyn. Syst. 39, 1190-1210 (2019)


\bibitem{Ar3} A. Artigue, \emph{Expansive flows of surfaces.} Disc. and Cont. Dyn. Sys., 33 (2013), 505-525.

\bibitem{Ar1} A. Artigue,
 \emph{Expansive Dynamical System.} Doctoral Thesis, UDELAR -2015

\bibitem{Ar4}
A. Artigue,
\emph{Singular cw-expansive flows}, Discrete and Continuous Dynamical Systems, 37(6): 2945-2956 (2017).


\bibitem{Ar2}
A. Artigue,
\emph{Rescaled expansivity and separating flows}, Discrete Contin. Dyn. Syst., 38 (2018), 4433-4447.




\bibitem{BW} R. Bowen, P. Walters, \emph{Expansive one-parameter flows.}
J. Diff. Equs. 12, 180-193 (1972)

\bibitem{CSM} B. Carrasco-Oliveira, B. San-Martin, \emph{On the $\mathcal{K}^*$-Expansiveness of the Rovella Attractor.}
Bulletin of the Brazilian Mathematical Society, New Series volume 48, 649-662 (2017).

\bibitem{Fa} A. Fathi, \emph{Expansiveness, hyperbolicity and Hausdorff dimension.} Comm. Math. Phys. 126 (1989), no. 2, 249-262

 \bibitem{Fl} L. Flinn.\emph{Expansive Flows.} PhD thesis, Warwick University, 1972

\bibitem{GSW} S. Gan, Y. Shi and L. Wen. \emph{On the singular hyperbolicity of star flows.} J. Mod. Dyn. 8:191–219, 2014.


\bibitem{Her}
J.R. Hertz, \emph{Continuum-wise expansive homeomorphisms on Peano continua.} Pre-Print,	arXiv:math/0406442 .

\bibitem{JNY} W. Jung, N.  Nguyen,Y.  Yang,
\emph{Spectral decomposition for rescaling expansive flows with rescaled shadowing.} Discrete and Continuous Dynamical Systems - A,
volume 40,  pp. 2267-2283, (2020)

\bibitem{Ka} H. Kato, \emph{Continuum-wise expansive homeomorphisms}, Canad. J. Math., 45 (1993), 576.

\bibitem{K1}
M. Komuro,
{\em Expansive properties of Lorenz attractors.} Theory of Dynamical Systems and Its Application to Nonlinear Problems, 4-26. World Sci. Publishing, Kyoto.

\bibitem{Koz} A. Kosinski 	\textit{Differential Manifolds.}
Pure and Applied Mathematics 138,	ACADEMIC PRESS, INC. , 2012.

\bibitem{KS}
H.B. Keynes, M. Sears,
\emph{Real-expansive flows and topological dimension.} Ergodic Theory and Dynamical Systems 1 (2), pp. 179-195.

\bibitem{L1}
S. Liao, 
\emph{Standard systems of differential equations.}
 Acta. Math. Sinica, 17 (1974),
100-109, 175-196, 270-295. (in Chinese)

\bibitem{L2} 
S. Liao, \emph{Qualitative Theory of Differentiable Dynamical Systems.}
 Science Press of China (1996). (in English)

\bibitem{Ma} R. Ma\~{n}\'{e}, \emph{Expansive homeomorphisms and topological dimension.} Trans. Amer. Math. Soc. 252 (1979), 313-319.

\bibitem{MM}
R.J. Metzger, C.A. Morales, \emph{The Rovella attractor is a homoclinic class.} Bull Braz Math Soc, New Series 37, 89-101 (2006).

\bibitem{P}
M. Paternain,  \emph{Expansive flows and the fundamental group.}  Bulletin of the Brazilian Mathematical Society, 24 (1993), 179-199.


\bibitem{PYY} M.J. Pac\'{i}fico, F. Yang, J. Yang  \emph{Entropy theory for sectional hyperbolic flows.} Ann. Inst. Henri
Poincar\'{e}, Anal. Non Lin\'{e}aire, 38 (2021), 1001–1030.

\bibitem{PYY2} M.J. Pac\'{i}fico, F. Yang, J. Yang  \emph{An entropy dichotomy for singular star flows.} Pre-print, arXiv:2101.09480

\bibitem{Rob} C. Robinson \emph{Dynamical systems. Stability, symbolic dynamics, and chaos.} Studies in Advanced Mathematics. CRC Press, Boca Raton, FL, 1995. xii+468 pp. ISBN: 0-8493-8493-1 

\bibitem{RWY} A. Rojas, X. Wen, Y. Yang, \emph{Sufficient conditions for rescaling expansivity}. Pre-print, arXiv:2311.18184.

\bibitem{Ro} A. Rovella. \emph{The dynamics of perturbations of the contracting Lorenz attractor}
Bull. Braz. Math. Soc. 24 (1993), 233-259

\bibitem{U} W.R. Utz,
\emph{Unstable homeomorphisms.}, Proc. Amer. Math. Soc., 1 (1950), 769.


\bibitem{WW} X. Wen, L. Wen,
\emph{ A rescaled expansiveness of flows}, Trans. Amer. Math. Soc., 371 (2019), 3179-3207.

\bibitem{Y} L.S. Young,\emph{ Entropy of continuous flows on compact 2-manifolds.} Topology 16 (1977), 469-471.










 
\end{thebibliography}
\end{document}